\documentclass[12pt]{amsart}
\numberwithin{equation}{section}
\usepackage{amsmath,amsthm,amsfonts,amscd,eucal}

\usepackage{xypic}
\usepackage{graphicx}

\usepackage{amssymb}
\def\cb{{\mathcal B}}

\def\ce{{\mathcal E}}
\def\cf{{\mathcal F}}

\def\ch{{\mathcal H}}

\def\cs{{\mathcal S}}

\def\ga{{\mathfrak A}} 
\def\gb{{\mathfrak B}}

\def\gam{{\mathfrak M}}

\def\gq{{\mathfrak Q}}

\def\gs{{\mathfrak S}}

\def\gx{{\mathfrak X}}

\def\bc{{\mathbb C}}

\def\bi{{\mathbb I}}

\def\bl{{\mathbb L}}
\def\bm{{\mathbb M}}
\def\bn{{\mathbb N}}
\def\bp{{\mathbb P}}
\def\br{{\mathbb R}}

\def\bz{{\mathbb Z}}

\def\a{\alpha}
\def\b{\beta}
\def\g{\gamma}  \def\G{\Gamma}
\def\d{\delta}

\def\l{\lambda} \def\L{\Lambda}

\def\s{\sigma} 
\def\t{\tau}
\def\f{\varphi}  \def\F{\Phi}
\def\th{\theta}  
\def\om{\omega} \def\Om{\Omega}

\def\id{\hbox{id}}
\def\ker{\hbox{Ker}}

\newtheorem{thm}{Theorem}[section]
\newtheorem{lem}[thm]{Lemma}

\newtheorem{prop}[thm]{Proposition}

\theoremstyle{definition}

\newtheorem{defin}[thm]{Definition}

\def\ad{\mathop{\rm ad}}

\def\idd{{1}\!\!{\rm I}}

\begin{document}

\title[Wick order]
{Wick order, spreadability and exchangeability for monotone commutation relations}
\author{Vitonofrio Crismale}
\address{Vitonofrio Crismale\\
Dipartimento di Matematica\\
Universit\`{a} degli studi di Bari\\
Via E. Orabona, 4, 70125 Bari, Italy}
\email{\texttt{vitonofrio.crismale@uniba.it}}
\author{Francesco Fidaleo}
\address{Francesco Fidaleo\\
Dipartimento di Matematica\\
Universit\`{a} degli studi di Roma Tor Vergata\\
Via della Ricerca Scientifica 1, Roma 00133, Italy} \email{{\tt
fidaleo@mat.uniroma2.it}}
\author{Maria Elena Griseta}
\address{Maria Elena Griseta\\
Dipartimento di Matematica\\
Universit\`{a} degli studi di Bari\\
Via E. Orabona, 4, 70125 Bari, Italy}
\email{\texttt{maria.griseta@uniba.it}}
\date{\today}

\begin{abstract}

\vskip0.1cm\noindent We exhibit a Hamel basis for the concrete $*$-algebra $\gam_o$ associated to monotone commutation relations realised on the monotone Fock space, mainly composed by Wick ordered words of annihilators and creators. We apply such a result to investigate spreadability and exchangeability of the stochastic processes arising from such commutation relations. In particular, we show that spreadability comes from a monoidal action implementing a dissipative dynamics on the norm closure $C^*$-algebra $\gam=\overline{\gam_o}$. Finally, we determine the structure of the set of exchangeable and spreadable monotone stochastic processes, by showing that both coincide with the stationary ones. 
\end{abstract}

\maketitle

\section{introduction}

\label{sec1}
The interest in noncommutative stochastic processes invariant under some distributional symmetries had a huge increase in the last years, and the references listed in \cite{CrFid2, EGK} offer a rich but not exhaustive account. This development is certainly related to the possibility of introducing de Finetti-type theorems (see {\it e.g.} \cite{CF0, K, Kr, St2} and the references cited therein), which possess natural and powerful applications in various fields such as quantum information theory and quantum statistical mechanics. In view of potential interest in the areas of quantum realm above cited, these notes give some new results for invariant stochastic processes coming from some exotic relations between creation and annihilation operators, as the monotone ones.

The investigation of strong ergodic properties for the so called monotone $C^*$-algebra started in \cite{CFL}, where stationary stochastic processes were classified too. A key argument for such achievements was to find a suitable family of generators for the underlying monotone $*$-algebra. Thus, one can naturally wonder if a more detailed study of the algebraic structure may offer a way to describe the monotone stochastic processes invariant under some further distributional symmetries, like spreadability and exchangeability. The reader is referred to \cite{Ka} for a treatise in the commutative setting.

More in detail, it appears convenient to achieve a linear basis for the monotone $*$-algebra, whose elements are put in a convenient way. This was done for Bose or Fermi annihilators and creators in the so called Wick or normal form, in honour of the Italian theoretical physicist Gian Carlo Wick ({\it cf.} \cite{W}), who firstly guessed the gain obtained after writing such products in a suitable way by using the Canonical (Anti)Commutation Relations. Natural applications took place in quantum field theory, and statistical mechanics of huge systems composed by particles obeying to Bose and Fermi statistics. In addition, it has been recorded an enormous use in managing the forthcoming standard models, including quantum electrodynamics and chromodynamics. The idea of the normal order consists in writing a generic word given by annihilators and creators as a linear combinations of words, where all creators appear to the left of the annihilators. The words in normal form then generate algebraically the involved $*$-algebra of operators, and allow a great simplification in computing the matrix elements of fields and observables.

Providing normal ordering to words of annihilators and creators looks a bit more complicated, with respect to the above cited cases, for some more general exotic commutation relations naturally appearing. As an example, we mention the computations in \cite{BKS} concerning the so called $q$-deformed commutation relations.

Nevertheless, in view to potential applications to various fields of mathematics and physics, it seems natural to address the question of putting in normal form products of annihilators and creators satisfying more complicated commutation relations, when possible.

The first goal of the present paper consists in finding  the normal form for monotone (or equivalently anti-monotone) commutation relations. It is performed by finding a basis with almost all elements given by Wick ordered words. Contrarily to other well known cases such as boson, fermion and even $q$-deformed, this feature is not easily predictable, since it does not directly come from a repeated use of the commutation relations.

Coming back to general features of quantum probability, very recently ({\it cf.} \cite{CrFid, CrFid2}) it was shown a one-to-one correspondence between unitarily equivalent classes of quantum stochastic processes on the index set $J$, for the sample $C^*$-algebra $\ga$, and states on the free product $C^*$-algebra $*_J\ga$. Suppose further that $g:J\to J$ is a general permutation of $J$ ({\it i.e.} a one-to-one map of $J$ onto $J$). Then a $*$-automorphism $\a_g$ of $*_J\ga$ is naturally induced by applying $g$ on indices in $J$. In this way, stochastic processes which are invariant after permuting the indices by the point action of $g$ correspond to those states on $*_J\ga$ which are invariant under the transposed action of $\a_g$ in a natural way. To be more precise, this means that exchangeable or stationary (provided $J$ coincides with $\bz$ for the latter) stochastic processes correspond to symmetric or shift invariant states on the free product $C^*$-algebra, respectively. As a consequence, the properties of stochastic processes, invariant or not under distributional symmetries, can be achieved studying the corresponding, invariant or not, states. This can be performed on any concrete $C^*$-algebra, seen as the quotient of the free product $C^*$-algebra via the universal property of the latter. This transfer entails the possibility to handle stationarity or exchangeability using standard results of ergodic theory, see {\it e.g.} \cite{BR1}. As an example, the ergodic decomposition of invariant states can be exploited to treat the quantum analogue of some celebrated results in classical probability, like de Finetti-type theorems. For a better understanding on the subject, the reader can compare the analysis in \cite{CF0, St2} with the content of Section 2 in \cite{CrFid2}.

Besides exchangeability and stationarity, there is another natural symmetry, called spreadability (see {\it e.g.} \cite{Ka}), inherited from classical probability which can be investigated in the setting of quantum stochastic processes. Very recently, spreadability has been described in some detail in \cite{EGK} from the viewpoint of category theory. Therefore, it appears natural to get a description of spreadable quantum stochastic processes in terms of invariance of states under the action of some natural algebraic object on the involved $C^*$-algebra. Here, there is the second goal of the present paper. By a direct application of the Wick order for monotone commutation relations, we show how spreadability is provided by a suitable monoid which acts by means of $*$-endomorphisms on the monotone $C^*$-algebra. Up to our knowledge, this seems the first construction in literature where such a symmetry comes out from an explicit implementation of a unital semigroup action.

After recalling some notions, in Proposition \ref{partial} of Section \ref{sec2} we remark that spreadable stochastic processes, or equivalently spreading invariant states, are those whose joint distributions are invariant under the natural action of all right and left hand side partial shifts. As a consequence, if the monoid algebraically generated by all of such partial shifts acts on the $C^*$-algebra by completely positive maps, the structure of spreadable stochastic processes is completely determined once one knows how to classify the relative invariant states.

The main topic of the present paper is the framework arising from the representation of the monotone commutation relations on the monotone Fock space, and then it deals with concrete unital $C^*$-algebra $\gam$ generated by monotone annihilators and creators. Once again, we recall that our monotone stochastic processes correspond precisely to elements in the set of states $\cs(\gam)$.

Section \ref{sec3} is devoted to determine a Hamel basis for the monotone $*$-algebra $\gam_o$ algebraically generated by monotone annihilators and creators. We will show that most of the elements therein are Wick ordered words. Although this result, not automatically induced by the anomalous commutation relations in the monotone setting, should possess an interest in itself for forthcoming potential applications in physics and information theory, as previously suggested, here we present some direct applications.

Indeed, in Section \ref{sec4} we exploit the basis to realise the above cited action of partial shifts on the monotone $C^*$-algebra. More in detail, we show that the monoid generated by such maps acts on $\gam_o$ as $*$- endomorphisms, which can be extended by continuity to the whole $\gam$. In the final part of the section, we classify the spreading invariant states and show they coincide with the stationary ones. They are indeed those connecting two ergodic states ({\it cf.} \cite{CFL}, Theorem 5.12), which are thus the extreme points of a segment.

In Section \ref{sec5}, we deal with the exchangeable monotone stochastic processes. As already pointed out in \cite{CFL}, the action of the permutation group $\bp_\mathbb{Z}$ here is not natural as in other well known cases like boson, fermionic, free or boolean, since $\bp_\bz$ does not naturally act as Bogoliubov automorphisms of the monotone algebra. However,
by using the Hamel basis previously found, for each permutation $\s\in\bp_\bz$ we can construct a map $T_\s:\gam_o\to\gam_o$ which implements $\s$. Unfortunately, such maps are not positive, and it seems a difficult task to show their boundedness. Thus, it is unknown if they extend to (bounded) maps on the whole monotone $C^*$-algebra.
As distributional symmetries deal only with the algebraic structure,
yet we have enough informations to provide the structure of exchangeable monotone states in Proposition \ref{symm}.

\section{preliminaries}
\label{sec2}

The following lines are mainly devoted to recall some notations and definitions frequently used in the sequel. If not otherwise specified, any $C^*$-algebra appearing below will be supposed unital. In addition, all morphisms are supposed to be $*$-morphisms.

\subsection{$C^*$-dynamical systems}
A $C^*$-dynamical system is the triplet $(\ga, M,\Gamma)$, where $\ga$ is a $C^*$-algebra with unit $\idd$, $M$ is a monoid, and finally $\G$ denotes a
representation $g\in M\mapsto \G_g$ of $M$ by completely positive identity preserving maps of $\ga$.

When we replace the monoid by a group, the $C^*$-dynamical system $(\ga, G, \a)$ is made of the unital $C^*$-algebra $\ga$, the group $G$, and the representation $\a$ of $G$ into the group of the $*$-automorphisms ${\rm Aut}(\ga)$ of $\ga$. In absence of bijections, one speaks of dissipative dynamics, whereas the automorphisms usually describe reversible, or hamiltonian whenever $G=\br$, flows, see {\it e.g.} \cite{BR1}.

Denote $\cs(\ga)$ the set of the states ({\it i.e.} the positive normalised functionals) of $\ga$. The part $\cs_M(\ga)\subset\cs(\ga)$ consists of the $M$-invariant states:
\begin{equation}
\label{v1}
\cs_M(\ga):=\big\{\f\in\cs(\ga)\mid\f\circ \G_g=\f,\,\,g\in M\big\}.
\end{equation}
It is $*$-weakly compact, and its extremal points $\ce_M(\ga):=\partial\cs_M(\ga)$
are called {\it ergodic states}.

Among the groups we deal with, we mention that of all finite permutations $\bp_J$ of an arbitrary index set $J$ given by
$$
\bp_J:=\bigcup\big\{\bp_\L\mid\L\,\text{finite subset of}\,J\big\},
$$
where $\bp_\L$ is the symmetric group associated to the finite set $\L$. We also mention the group generated by the one step shift $\t(i):=i+1$
of the integers $\bz$, which is isomorphic to $\bz$ itself.

We also consider the monoid $(\bl_\bz,\circ)$ consisting of all strictly increasing maps $g:\bz\to\bz$ under the composition operation, and its submonoid $(\bi_\bz,\circ)$ generated by the compositions of all partial shifts, see Section \ref{scfpr}. Very often, for the involved maps $f,g:\bz\to\bz$, we drop the composition symbol simply by writing that as a product: $f\circ g\equiv fg$. Accordingly, we also write the relative monoids without pointing out the composition. For example, $(\bl_\bz,\circ)$ will be denoted simply as $\bl_\bz$.

\subsection{Stochastic processes}
\label{sec2.2}
Recall that for a given set $J$ and unital $C^*$-algebras $\{\ga_j\}_{j\in J}$, their unital free product $C^*$-algebra $\ast_{j\in J} \ga_j$ is the unique unital $C^*$-algebra, together with  unital monomorphisms $i_j:\ga_j\rightarrow \ast_{j\in J} \ga_j$ such that for any unital $C^*$-algebra $B$ and unital morphisms $\F_j:\ga_j\rightarrow B$, there exists a unique unital homomorphism $\F:\ast_{j\in J} \ga_j\rightarrow B$ making commutative the following diagram
\begin{equation*}
\xymatrix{ \ga_j \ar[r]^{i_j} \ar[d]_{\F_{j}} &
\ast_{j\in J} \ga_j \ar[dl]^{\quad\F\quad\quad \quad\quad.} \\
B}
\end{equation*}
In the present paper, we always deal with unital free product $C^*$-algebras based on a single unital $C^*$-algebra $\ga$, the algebra of the samples, called {\it the free product
$C^*$-algebra}, and denoted simply as
$\ast_{J} \ga$. We refer the reader to \cite{VDN, CrFid} for further details.

We note that in the case $J=\bz$, the group $\bp_\bz$ as well as the group $\bz$ generated by the one step shift, act in a natural way on $\ast_{\bz} \ga$.

A (quantum) {\it stochastic process}, labelled by the index set $J$ and determined up to unitary equivalence, is a quadruple
$\big(\ga,\ch,\{\iota_j\}_{j\in J},\Om\big)$, where $\ga$ is a $C^{*}$-algebra, $\ch$ is an Hilbert space,
the $\iota_j$ are $*$-homomorphisms of $\ga$ in $\cb(\ch)$, and
$\Om\in\ch$ is a unit vector, cyclic for  the von Neumann algebra
$M:=\bigvee_{j\in J}\iota_j(\ga)$ naturally acting on $\ch$.

Fix $n\in\bn$, $j_1,\ldots, j_n\in J$ with different contiguous indices, and  $A_{j_1},\ldots,A_{j_n}\in\ga$. The joint probability $p$
of the stochastic process for the element $A_{j_1}*\cdots*A_{j_n}\in*_J\ga$
is given by
$$
p\big(A_{j_1},\ldots,A_{j_n}\big)=\big\langle\iota_{j_1}(A_{j_1})\cdots\iota_{j_n}(A_{j_n})\Om, \Om\big\rangle\,.
$$
In Theorem 2.3 of \cite{CrFid2}, it was proven that there is a one-to-one correspondence between states on $\ast_{J} \ga$ and quantum stochastic processes. More in detail, one sees that the quadruple $\big(\ga,\ch,\{\iota_j\}_{j\in J},\Om\big)$ determines a unique state $\f\in\cs\big(\ast_{J}\ga\big)$, and a representation $\pi$ of $\ast_{J} \ga$ on the Hilbert space $\ch$ such that $(\pi,\ch,\Om)$ is the Gelfand-Naimark-Segal (GNS fir short) representation of the state $\f$. Conversely, each state $\f\in\cs\big(\ast_{J}\ga\big)$ defines a unique stochastic process, up to unitary equivalence, just by looking at its GNS representation. The interested reader is addressed to \cite{CrFid2} for more details on quantum stochastic processes, their joint probabilities and relations with the classical ones.

A stochastic process is said to be {\it exchangeable} if, for each $g\in\bp_J$, $n\in\bn$, $j_1,\ldots, j_n\in J$, $A_1,\ldots, A_n\in\ga$,
\begin{equation}
\label{exzc}
\langle\iota_{j_1}(A_1)\cdots\iota_{j_n}(A_n)\Om,\Om\rangle
=\langle\iota_{g(j_{1})}(A_1)\cdots\iota_{g(j_{n})}(A_n)\Om,\Om\rangle\,.
\end{equation}
Suppose that $J=\bz$. The process is said to be {\it stationary} if  for each $n\in\bn$, $j_1,\ldots j_n\in\bz$, $A_1,\ldots A_n\in\ga$,
$$
\langle\iota_{j_1}(A_1)\cdots\iota_{j_n}(A_n)\Om,\Om\rangle
=\langle\iota_{j_{1}+1}(A_1)\cdots\iota_{j_{n}+1}(A_n)\Om,\Om\rangle\,.
$$
Fix now the monoid $\bl_\bz$.
It was argued in \cite{EGK} that it is the starting point to manage the quantum generalisation of the notion of spreadability for commutative stochastic processes ({\it cf.} \cite{Ka}).
Indeed, a stochastic process is then said to be {\it spreadable} if for each $n\in \mathbb{N}$, $j_1,j_2,\ldots, j_n$ in $\mathbb{Z}$, $A_1,\ldots, A_n\in\ga$, and $g\in\bl_\bz$, one has
\begin{equation}
\label{spre}
\langle\iota_{j_1}(A_1)\cdots\iota_{j_n}(A_n)\Om,\Om\rangle
=\langle\iota_{g(j_{1})}(A_1)\cdots\iota_{g(j_{n})}(A_n)\Om,\Om\rangle\,.
\end{equation}
From now on, we take $\bz$ as the common index set and note that exchangeability and stationarity are respectively the strongest and the weakest among them. Moreover, contrarily to the classical case, spreadability in general does not imply exchangeability, see \emph{e.g.} \cite{K}.

After arguing that stochastic processes labelled by the index set $J=\bz$, and involving the algebra of the sample $\ga$, correspond to states on the free product algebra
$\ast_{J} \ga$, we also note ({\it cf.} Theorem 2.3 of \cite{CrFid2}) that exchangeable or stationary stochastic processes give rise to {\it symmetric} or {\it shift invariant} states on
$\ast_{J} \ga$, respectively.

As noticed in Remark 3.5 of \cite{CrFid}, some particular classes of processes, like those arising from the Canonical Commutations or Anticommutation Relations, are associated to classes of states coming from a quotient $\gq=\ast_{J} \ga/\ker(\pi)$, $\pi:\ast_{J}\ga\to\gq$ being a suitable $*$-epimorphism. Therefore, it would be of interest to study the invariance properties of such processes directly on the $C^*$-algebra $\gq$. Following this line, we shall see below that, for the monotone case, spreadability is indeed associated to a representation of an appropriate monoid on the monotone $C^*$-algebra $\gam$.

\subsection{Spreadability}
\label{scfpr}

The following preliminary definitions and results will be often used, and allow to manage spreadable stochastic processes in a suitable way in the sequel.

Let us take $h\in\mathbb{Z}$. The {\it right hand side partial shift} based on $h$ is the one-to-one map $\theta_h:\mathbb{Z}\rightarrow \mathbb{Z}$ such that
$$
\theta_h(k):=\left\{\begin{array}{ll}
                      k & \text{if}\,\, k<h\,,\\
                      k+1 & \text{if}\,\, k\geq h\,.
                    \end{array}
                    \right.
$$
The {\it left hand side partial shift} based on $h$ is the one-to-one $\psi_h:\mathbb{Z}\rightarrow \mathbb{Z}$ such that
$$
\psi_h(k):=\left\{\begin{array}{ll}
                      k & \text{if}\,\, k>h\,, \\
                      k-1 & \text{if}\,\, k\leq h\,.
                    \end{array}
                    \right.
$$
Denote $\bi_\bz\subset \bl_\bz$ the submonoid generated by all forward and backward partial shifts $\{\th_h\}_{h\in\bz}$ and $\{\psi_h\}_{h\in\bz}$. As all powers $\t^k$, $k\in\bz$ of the shift on $\bz$ are contained in $\bl_\bz$, it is easy to check the following relations
\begin{align}
\begin{split}
\label{pbfs}
&\t^k\theta_l\t^{-k}=\theta_{k+l}\,,\\
&\t^k\psi_l\t^{-k}=\psi_{k+l}\,,\,\,k,l\in\bz\,,
\end{split}
\end{align}
exploited later on.

From now on, if $m,n\in\bz$ and $m<n$, the subset $\{m,m+1,\ldots, n\}$ will be simply denoted by $[m,n]$. The next result relates the action of increasing maps on $[m,n]\subset\mathbb{Z}$ with the actions of partial shifts on the same subsets.
\begin{prop}
\label{partial}
For each finite subset $[m,n]\subset\bz$ and $l:\bz\to\bz$ defining a subsequence of $\bz$, there exists $r=r_{[m,n],l}\in\bi_\bz$ such that
$l([m,n])=r([m,n])$.
\end{prop}
\begin{proof}
We distinguish two cases. First, $l(m)\geq m$. In this case,
$$
r=\th_{l(n-1)+1}^{l(n)-l(n-1)-1}\cdots\th_{l(m)+1}^{l(m+1)-l(m)-1}\th_m^{l(m)-m}
$$
is one of such elements $r_{[m,n],l}$ of $\bi_\bz$.

Suppose now $l(m)<m$. One first applies $\psi_n^{m-l(m)}$ to $[m,n]$, obtaining
$$
\psi_n^{m-l(m)}([m,n])=[l(m),l(m)+n-m]\,.
$$
Then one iterates the same arguments as above to $[l(m),l(m)+n-m]$.
\end{proof}
In view of \eqref{spre} and Proposition \ref{partial}, spreadable stochastic processes are exactly those whose joint distributions are invariant under all partial shifts. In addition,
suppose for example that the partial shifts act on our model $C^*$-algebra $\gb$ by $*$-endomorphisms
$$
g\in \mathbb{I}_\mathbb{Z}\mapsto \G_g\in \text{End}(\gb)\,.
$$
In this case, a state $\f$ is said {\it spreading invariant} if as usual, for each $g\in\mathbb{I}_\mathbb{Z}$, one has $\f\circ \G_g= \f$. Thus, a spreadable stochastic process uniquely arises from a spreading invariant state.

\noindent Suppose that also the permutations $\bp_\bz$ act as a group of $*$-automorphisms on $\gb$. We note that, if one takes a finite part $[m,n]\subset\bz$ and an element $g\in\bi_\bz$, there exists a (not unique) element $\s_{[m,n],g}\in\bp_\bz$ such that
\begin{equation*}
\s_{[m,n],g}([m,n])=g([m,n])\,.
\end{equation*}
In addition, for the shift $\t$ one has $\t([m,n])=\th_j([m,n])$ for each $j\leq m$.

As a consequence of the previous discussion, it follows for the $*$-weakly compact convex sets of states as in \eqref{v1}, and describing exchangeable, spreadable and stationary processes,
\begin{equation}
\label{scaprz}
\cs_{\bp_\bz}(\gb)\subset\cs_{\bi_\bz}(\gb)\subset\cs_{\bz}(\gb)\,.
\end{equation}

\section{a basis for the monotone $*$-algebra}
\label{sec3}

We begin by recalling some useful features, the reader being addressed to \cite{CFL,CFL2, Lu, Mur} for further details.

For $k\geq 1$, denote $I_k:=\{(i_1,i_2,\ldots,i_k) \mid i_1< i_2 < \cdots <i_k, i_j\in \mathbb{Z}\}$. The discrete monotone Fock space is the Hilbert space $\cf_m:=\bigoplus_{k=0}^{\infty} \ch_k$, where for any $k\geq 1$, $\ch_k:=\ell^2(I_k)$, and $\ch_0=\mathbb{C}\Om$, $\Om$ being the Fock vacuum. Borrowing the terminology from the physical language, we call each
$\ch_k$ the $k$th-particle space and denote by $\cf^o_m$ the total set of finite particle vectors in $\cf_m$, that is
$$
\cf^o_m:=\bigg\{\sum_{n=0}^{\infty} c_n\xi_n \mid \xi_n\in \ch_n,\,\, c_n\in \mathbb{C}\,\,\, \text{s.t.}\,\, c_n=0\,\,\, \text{but a finite set}  \bigg\} \,.
$$

Let $(i_1,i_2,\ldots,i_k)$ be an increasing sequence of integers. The generic element of the canonical basis of $\cf_m$ is denoted by $e_{(i_1,i_2,\ldots,i_k)}$, with the convention $e_\emptyset=\Om$. Very often, we write $e_{(i)}=e_i$ to simplify the notations.
The monotone creation and annihilation operators are respectively given, for any $i\in \mathbb{Z}$, by
\begin{equation}
\label{crea}
a^\dagger_i e_{(i_1,i_2,\ldots,i_k)}:=\left\{
\begin{array}{ll}
e_{(i,i_1,i_2,\ldots,i_k)} & \text{if}\, i< i_1\,, \\
0 & \text{otherwise}\,, \\
\end{array}
\right.
\end{equation}
\begin{equation}
\label{anni}
a_ie_{(i_1,i_2,\ldots,i_k)}:=\left\{
\begin{array}{ll}
e_{(i_2,\ldots,i_k)} & \text{if}\, k\geq 1\,\,\,\,\,\, \text{and}\,\,\,\,\,\, i=i_1\,,\\
0 & \text{otherwise}\,. \\
\end{array}
\right.
\end{equation}
One can check that both $a^\dagger_i$ and $a_i$ have unital norm (see {\it e.g.} \cite{Boz}, Proposition 8), are mutually adjoint, and satisfy the following relations
\begin{equation*}
\begin{array}{ll}
  a^\dagger_ia^\dagger_j=a_ja_i=0 & \text{if}\,\, i\geq j\,, \\
  a_ia^\dagger_j=0 & \text{if}\,\, i\neq j\,.
\end{array}
\end{equation*}
In addition, the following commutation relation
\begin{equation}
\label{iden}
a_ia^\dagger_i=I-\sum_{k\leq i}a^\dagger_k a_k
\end{equation}
is also satisfied, where the sum is meant in the strong operator topology, see \cite{CFL}, Proposition 3.2. We point out the following facts. First, \eqref{iden} is merely a consequence of the definitions \eqref{crea} and \eqref{anni}. Second, it falls in a wider class of "commutation relations" as established in Corollary 3.2 of \cite{BS}, without determining the structure of the algebras that generate them.

We denote by $\gam$ and $\gam_o$ the concrete unital
$C^*$-algebra, together with its dense unital $*$-algebra generated by the annihilators $\{a_i\mid i\in\mathbb{Z}\}$, acting on the monotone Fock space. Both algebras are canonically $*$-subalgebras of $\cb(\cf_m)$.
\begin{defin}
\label{deflampi}
A word $X$ in $\gam_o$ is said to have a $\lambda$-\textbf{form} if there are $m,n\in\left\{  0,1,2,\ldots
\right\}$ and $i_1<i_2<\cdots < i_m, j_1>j_2> \cdots > j_n$ such
that
$$
X=a_{i_1}^{\dagger}\cdots a_{i_m}^{\dagger} a_{j_1}\cdots a_{j_n}\,,
$$
with $X=I$, that is the empty word if $m=n=0$. Its length is $l(X)=m+n$.

A word $X$ is said to have a $\pi$-\textbf{form} if there are $m,n\in\left\{0,1,2,\ldots
\right\}$, $k\in\mathbb{Z}$, $i_1<i_2<\cdots < i_m, j_1>j_2> \cdots > j_n$ such that $i_m<k>j_1$ and
$$
X=a_{i_1}^{\dagger}\cdots a_{i_m}^{\dagger} a_{k}a_{k}^{\dagger} a_{j_1}\cdots a_{j_n}\,.
$$
The length is now $l(X)=m+2+n$.
\end{defin}
\vskip.3cm
As explained by lemmata in Section 5 of \cite{CFL}, any word in $\gam_o$ can be expressed by words in $\l$-form or in $\pi$-form. Thus, each element in this $*$-algebra is a finite linear combination of $\l$-forms and $\pi$-forms. We are going to prove that the linear structure of
$\gam_o$ provides a way to drastically reduce such generators, by producing a genuine basis. Almost all of the elements in this basis are in fact words with the creators on the left and annihilators on the right hand side. In other words, we are yielding a unique expression for each element of $\gam_o$ where summands are mostly in the so called {\it Wick}, or {\it normal} {\it form}.

Let $\L$ be the index set such that $\{X_\l\}_{\l\in\L}$ gives the totality of the $\l$-forms, that is
\begin{equation}
\label{xlambda}
X_\l=a_{i_1^{(\l)}}^{\dagger}\cdots a_{i_{m(\l)}^{(\l)}}^{\dagger} a_{j_1^{(\l)}}\cdots a_{j_{n(\l)}^{(\l)}}\,,
\end{equation}
for $i_1^{(\l)}<i_2^{(\l)}<\cdots < i_{m(\l)}^{(\l)}, j_1^{(\l)}>j_2^{(\l)}> \cdots > j_{n(\l)}^{(\l)}$, $m(\l),n(\l)\geq 0$, where, as usual, $m(\l)=n(\l)=0$ corresponds to the identity $I$. Since all of the $\l$-forms like $a^\dag_ia_i$ are in one-to-one correspondence with $\bz$, then $\bz\subset\L$  in a natural way. After this identification, we put $\G:=\L\setminus\bz$.

The following results are the main tools to construct our basis.
\begin{lem}
\label{linind}
The family $\{X_\l\}_{\l\in\G}\bigcup\{a_ia^\dag_i\}_{i\in\bz}\subset\gam_o$ is a linearly independent set.
\end{lem}
\begin{proof}
Let us take the finite linear combination
$$
X:=\sum_{\l\in H} \b_\l X_\l + \sum_{k\in K} \gamma_ka_ka^\dag_k\,,
$$
where $H$ and $K$ are finite parts of $\G$
and $\mathbb{Z}$, respectively.

We first suppose there exists $\l_o\in H$ such that $X_{\l_o}=I$. In this case, $X=0$ entails $\b_{\l_o}=0$. Indeed, for any $\l\in H \backslash \{\l_o\}$ and $X_\l$ as in \eqref{xlambda}, we take
$$
h_\l:=\left\{\begin{array}{cc}
                j_{n(\l)}^{(\l)} & \text{if}\,\,\, n(\l)>0\,, \\
                i_{m(\l)}^{(\l)} & \text{if}\,\,\, n(\l)=0\,,
              \end{array}
              \right.
$$
and denote $p_o:=\min_{\l\in H \backslash \{\l_o\}}\{h_\l\}$, $k_o:=\min\{k\in K\}$. From \eqref{crea} and \eqref{anni}, one has
$$
\b_{\l_o}=\langle Xe_{(p_o\wedge k_o)-1},e_{(p_o\wedge k_o)-1}\rangle=0\,,
$$
where $p_o\wedge k_o:=\min\{p_o,k_o\}$.

Thus, from now on we can suppose $l(X_\l)>0$ for any $\l\in H$. Here, one has $X_\l e_k\perp e_k$ for each $k\in\mathbb{Z}$, since we excluded all of the $\l$-forms $a_i^\dagger a_i$, $i\in\bz$.

Let $k_1<\cdots <k_r$ be the sequence of the elements of $K$ whenever $K\neq\emptyset$, and put $\xi_i:=e_{k_i+1}$, $i=1,\dots,r$. The orthogonality condition above and \eqref{crea}, \eqref{anni} give
\begin{equation*}
\langle X\xi_i, \xi_i\rangle
=\bigg\langle\sum_{p=1}^i\gamma_{k_p} a_{k_p}a^\dag_{k_p}\xi_i, \xi_i\bigg\rangle
=\sum_{p=1}^i\gamma_{k_p}\,,\quad i=1,\dots,r\,.
\end{equation*}
Thus, when $X=0$ we get $0=\langle X\xi_1, \xi_1\rangle=\gamma_{j_1}$, and then
$$
0=\langle X\xi_i, \xi_i\rangle=\gamma_{k_i}\,,\quad i=1,\dots,r\,,
$$
after an elementary iteration procedure.
Consequently, when $X$ equals to zero, one reduces the matter to $\sum_{\l\in H} \b_\l X_\l=0$. In this case, take an arbitrary $\l\in H$. For
$\xi_{\l}:=e_{j_{n(\l)}^{(\l)}}\otimes \cdots \otimes e_{j_{1}^{(\l)}}$, $\eta_{\l}:=e_{i_{1}^{(\l)}}\otimes \cdots \otimes e_{i_m^{(\l)}}$ and $X_\l$ as in \eqref{xlambda}, with the convention that $\xi_{\l}=\Om$ and $\eta_{\l}=\Om$ when $n(\l)=0$ or $m(\l)=0$, it follows
$$
0=\langle X \xi_{\l},\eta_{\l}\rangle=\langle \b_{\l}X_\l \xi_{\l},\eta_{\l}\rangle=\langle \b_{\l} \Om,\Om\rangle=\b_{\l}\,.
$$
\end{proof}
The linear operations reducing the algebraic generators of $\gam_o$ not included in the collection presented in the previous lemma, are listed below.
\begin{lem}
\label{sigen}
The following identities hold true.

(a) For $m,n\geq 1$, $k\in\mathbb{Z}$, $i_1<i_2<\cdots < i_m, j_1>j_2> \cdots > j_n$ such that $i_m<k>j_1$, one has
\begin{align}
\begin{split}
\label{pi1}
&a_{i_1}^{\dagger}\cdots a_{i_m}^{\dagger} a_{k}a_{k}^{\dagger} a_{j_1}\cdots a_{j_n}\\
=&\,a_{i_1}^{\dagger}\cdots a_{i_m}^{\dagger}a_{j_1}\cdots a_{j_n}
- \sum_{l=(i_m\vee j_1)+1}^k a_{i_1}^{\dagger}\cdots a_{i_m}^{\dagger} a_{l}^{\dagger}a_{l} a_{j_1}\cdots a_{j_n}\,,
\end{split}
\end{align}
where $i_m\vee j_1:=\max\{i_m,j_1\}$.

(b) For each $i\in\mathbb{Z}$, one has
$$
a^\dag_ia_i=a_{i-1}a^\dag_{i-1} - a_ia^\dag_i\,.
$$
\end{lem}
\begin{proof}
As previously noticed, the commutation rule \eqref{iden} is meaningful in the strong operator topology. Thus, we can freely use it as an identity of matrix elements between vectors in the monotone Fock space $\cf_m$.

In order to prove (a), consider $m,n\geq 1$, $k\in\mathbb{Z}$, $i_1<i_2<\cdots < i_m, j_1>j_2> \cdots > j_n$ such that $i_m<k>j_1$. One achieves \eqref{pi1}, as \eqref{iden} gives
\begin{align*}
&a_{i_1}^{\dagger}\cdots a_{i_m}^{\dagger} a_{k}a_{k}^{\dagger} a_{j_1}\cdots a_{j_n}\\
=&a_{i_1}^{\dagger}\cdots a_{i_m}^{\dagger} a_{j_1}\cdots a_{j_n}-\sum_{l\leq k}a_{i_1}^{\dagger}\cdots a_{i_m}^{\dagger}a^\dag_l a_l a_{j_1}\cdots a_{j_n} \\
=&a_{i_1}^{\dagger}\cdots a_{i_m}^{\dagger} a_{j_1}\cdots a_{j_n}- \sum_{l=(i_m\vee j_1)+1}^k a_{i_1}^{\dagger}\cdots a_{i_m}^{\dagger} a_{l}^{\dagger}a_{l} a_{j_1}\cdots a_{j_n}\,,
\end{align*}
where the last equality comes from \eqref{crea} and \eqref{anni}.

Take now $i\in \mathbb{Z}$. The identity in (b) follows after applying again \eqref{iden}, that is
$$
a_{i-1}a^\dag_{i-1} - a_ia^\dag_i=I-\sum_{l\leq i-1}a^\dag_l a_l - \bigg(I-\sum_{l\leq i}a^\dag_l a_l\bigg)= a^\dag_ia_i\,.
$$
\end{proof}
The following result ensures that, apart from the so-called trivial $\pi$-forms (\emph{i.e.} $\{a_ia^\dag_i\}_{i\in \mathbb{Z}}$), all the previous words are Wick ordered and, together with such trivial $\pi$-forms, provide a basis for the algebraic part of the monotone $C^*$-algebra.
\begin{thm}
\label{basis}
The families $\{X_\l\}_{\l\in\G}$ and $\{a_ia^\dag_i\}_{i\in \mathbb{Z}}$ form a Hamel basis of $\gam_o$.
\end{thm}
\begin{proof}
We have proven in Lemmata 5.5 and 5.6 of \cite{CFL} that $\gam_o$ is linearly generated by all the words consisting of $\l$ and $\pi$-forms.
On the other hand, the linear structure of such an algebra yields a further reduction for the elements belonging to these collections (Lemma \ref{sigen}). The remaining part is precisely given by the above mentioned families of linear independent elements ({\it cf.} Lemma \ref{linind}).
\end{proof}
Here, we report
some identities which take into account the algebraic structure of $\gam_o$. Let us consider
$$
X_1:=a^\dag_{i_1}\cdots a^\dag_{i_m}a_{j_1}\cdots a_{j_s}
$$
for $i_1<\cdots < i_m$, $j_1>\cdots > j_s$, and
$$
X_2:=a^\dag_{k_1}\cdots a^\dag_{k_r}a_{l_1}\cdots a_{l_p}
$$
for $k_1<\cdots < k_r$, $l_1>\cdots > l_p$. We denote
\begin{equation}
\label{ggta}
\d_{j)}(h):=\left\{
\begin{array}{ll}
  1 & \text{if}\,\,\,\,\, h<j\,, \\
  0 & \text{if}\,\,\,\,\, h\geq j\,.
\end{array}
\right.
\end{equation}
According to whether $s<r$ or $r\leq s$, from \eqref{crea} and \eqref{anni} one has
\begin{align}
\begin{split}
\label{m1}
&X_1X_2=\prod_{h=1}^s\d_{j_h,k_{s-h+1}}\d_{k_{s+1})}(i_m)a^\dag_{i_1}\cdots a^\dag_{i_m}a^\dag_{k_{s+1}}\cdots a^\dag_{k_r}a_{l_1}\cdots a_{l_p}\,,\\
&X_1X_2=\prod_{h=1}^r\d_{j_{s-h+1},k_h}\d_{j_{s-r)}}(l_1)a^\dag_{i_1}\cdots a^\dag_{i_m}a_{1}\cdots a_{j_{s-r}}a_{l_1}\cdots a_{l_p}\,,
\end{split}
\end{align}
respectively. Take now, for $i,j\in\mathbb{Z}$, $Y_1:=a_ia^\dag_i$ and $Y_2:=a_ja^\dag_j$. From Lemma 5.4 in \cite{CFL} and \eqref{pi1}, it follows
\begin{equation*}
X_1Y_1=\left\{
\begin{split}
\begin{array}{lll}
               \d_{j_s)}(i)X_1 & \text{if}\,\,\,\,\, s>0\,, \\
               X_1 & \text{if}\,\,\,\,\, s=0\,\,, i\leq i_m\,, \\
               X_1-\sum_{l=i_m+1}^i a^\dag_{i_1}\ldots a^\dag_{i_m}a^\dag_la_l & \text{if}\,\,\,\,\, s=0\,\,, i> i_m\,,
             \end{array}
\end{split}
\right.
\end{equation*}
\begin{equation}
Y_1X_1=\left\{
\begin{split}
\label{m4}
\begin{array}{lll}
               \d_{i_1)}(i)X_1 & \text{if}\,\,\,\,\, m>0\,, \\
               X_1 & \text{if}\,\,\,\,\, m=0\,\,, i\leq j_1\,, \\
               X_1-\sum_{l=j_1+1}^i a^\dag_l a_l a_{j_1}\ldots a_{j_n} & \text{if}\,\,\,\,\, s=0\,\,, i> j_1\,,
             \end{array}
\end{split}
\right.
\end{equation}
$$
Y_1Y_2=a_la^\dag_l\,,
$$
where $l:=\max\{i,j\}$. The description of the $*$-operation is elementary.

We end with the following description of the set $\G\cup \bz$. Indeed,
one denotes each element of the Hamel basis of $\gam_o$ by $X_\l$, $\l=(\l_1,\l_2)$, where $\l_1,\l_2\in2^\bz$ are all finite ascending ordered sets, including the empty set. Notice that the identity corresponds to $(\emptyset,\emptyset)$. The words of length 1 correspond to $a^\dagger_i=X_{(\{i\},\emptyset)}$, and $a_i=X_{(\emptyset,\{i\})}$.
Moreover, if $X_\l$ is a word of length 2 with $\l=(\{i\},\{j\})$, then
\begin{align*}
&i\neq j\,\,\text{corresponds to}\,\, X_\l=a^\dagger_ia_j\,,\\
&i= j\,\,\text{corresponds to}\,\, X_\l=a_ia^\dagger_i\,.
\end{align*}
The remaining cases correspond to words in $\l$-form of length at least $2$ in the following way. If $\l=(\{i_1,\cdots i_m\},\{j_n,\cdots j_1\})$, then
$$
X_\l=a^\dagger_{i_1}\cdots a^\dagger_{i_m}a_{j_1}\cdots a_{j_n}\,.
$$
In all cases, if $f:\bz\to\bz$ is order preserving (and in particular belongs to the monoid $\bl$), then $f$ acts componentwise on the index set: $f(\l):=\big(f(\l_1),f(\l_2)\big)$.

\section{the action of the partial shifts and spreadability}
\label{sec4}

As an application of the previous results, in this section we aim to show that all partial shifts on $\mathbb{Z}$, and consequently the monoid $\bi_{\bz}$ they generate, act as isometric
$*$-endomorphisms on $\gam$, and in addition characterise the invariant states with respect to such an action.

We start by defining the action of $\bi_\bz$ on the Hamel basis in Theorem \ref{basis} for $\gam_o$. Indeed, on each element $X_\l$ with
$\l=(\l_1,\l_2)$, we put
\begin{equation}
\label{beta1}
\b^{o}_k(X_\l):=X_{\th_k(\l)}\,,\quad \g^{o}_k(X_\l):=X_{\psi_k(\l)}\,,\quad k\in\bz\,.
\end{equation}
Since $\th_k$ and $\psi_k$ respect the natural order on $\mathbb{Z}$, by Theorem \ref{basis} such maps uniquely extend to $\gam_o$ by linearity.

As a first step, we manage $\th:=\th_0$ and $\psi:=\psi_0$ with the associated maps $\b^{o}_0$ and $\g^{o}_0$.
In the next lines, we verify that $\b_0^o$ (and similarly $\g_0^o$) preserves the algebraic structure, \emph{i.e.} they are
$*$-endomorphisms of $\gam_o$, whence they extend to $*$-endomorphisms on $\gam$.
For the convenience of the reader, we report some useful notions, perhaps well known to the experts, and refer them to \cite{P}.

Let $\gs\subset\ga$ be a subspace, not necessarily closed, of the unital $C^*$-algebra $\ga$. If $\idd\in\gs=\gs^*$, it is called an {\it operator system}. A linear map $A:\gs\to\gx$ from $\gs$ to the normed space $\gx$ is said {\it completely bounded} if the norms of
\begin{equation}
\label{cpc}
A\otimes\id_{\bm_n(\bc)}:\bm_n(\gs)\to\bm_n(\gx)\,,\quad n=1,2,\dots
\end{equation}
are uniformly bounded. We put
$$
\|A\|_{\rm cb}:=\sup\big\{\big\|A\otimes\id_{\bm_n(\bc)}\big\|\mid  n=1,2,\dots\big\}\,.
$$
$A$ is said {\it completely positive} if all maps in \eqref{cpc} are positive, provided $\gx$ is also an operator system.

The following result is crucial in the sequel. We report it for the convenience of the reader, referring to Proposition 3.6 in  \cite{P} for its proof.
\begin{prop}
\label{paul}
Let $A:\gs\to\gb$ be a completely positive map between the operator system $\gs$ and the, not necessarily unital, $C^*$-algebra $\gb$. Then $A$ is completely bounded with
$$
\|A\|_{\rm cb}=\|A(\idd)\|=\|A\|\,.
$$
\end{prop}
The key point for the forthcoming analysis is the proof of the following statement.
\begin{prop}
\label{prcrzq}
The map $\b^{o}_0$ provides a $*$-endomorphism of $\gam_o$.
\end{prop}
\begin{proof}
Since the $*$-operation is easily preserved, we restrict the matter to the product. By looking at \eqref{m1}-\eqref{m4}, the thesis follows if one checks
$$
\d_{k,l}=\d_{\th(k),\th(l)}\,,\quad \d_{k)}(l)=\d_{\th(k))}(\th(l))\,,
$$
for the Kronecker symbol and the function given in \eqref{ggta}, respectively.
Since $\th$ is one-to-one and order preserving, the assertion is achieved straightforwardly.
\end{proof}
\begin{prop}
\label{prcrzq1}
The map $\b^{o}_0$ extends to a $*$-endomorphism $\b_0$ of $\gam$.
\end{prop}
\begin{proof}
Since a $*$-morphism is automatically completely positive ({\it cf.} \cite{P}), Proposition \ref{paul} tells us that $\b^o_0$ extends to a bounded linear map $\b_0$ on the whole $\gam$. Since the product and the involution are continuous in the norm topology, $\b_0$ results to be a $*$-endomorphism.
\end{proof}
We remark that the above analysis applies {\it mutatis-mutandis} to backward partial shift $\psi$, thus giving $\g_0$.
Finally, we also recall that the shift $\a$ on $\gam$ is unitarily implemented by $\a(X)=UXU^*$, where
$$
Ue_B:=e_{\t(B)}\,,
$$
$e_B$ being a generic element of the canonical orthonormal basis of $\cf_m$.
This allows us to obtain the actions of all the generators $\{\th_k,\psi_l\mid k,l\in\bz\}$ of $\bi_\bz$.

Summarising, here there is the main result of the present section.
\begin{thm}
\label{vw3}
For $k\in\bz$, after defining on $\gam$
$$
\b_k:=\ad(U^k)\circ\b_0\circ\ad(U^{-k})\,,\quad \g_k:=\ad(U^k)\circ\g_0\circ\ad(U^{-k})\,,
$$
it is established an action of the monoid $\bi_\bz$ on $\gam$ by $*$-endo\-morphisms.
\end{thm}
\begin{proof}
We first notice that the $\th_k$ and $\psi_k$ generate $\bi_\bz$. Furthermore, Proposition \ref{prcrzq1} gives $\b_0$ and $\g_0$ are $*$-endomorphisms of the whole $\gam$. By \eqref{pbfs}, the above formula provides the extension of all maps defined in \eqref{beta1} to $\gam$, giving an action of all generators of $\bi_\bz$ on $\gam$ by $*$-endomorphisms.
\end{proof}
Following the same lines as above (and in particular by using Proposition \ref{paul}), one notes it can be proven that
the monoid $\bi_\bz$ acts by $*$-endomorphisms of the free product $C^*$-algebra $*_\bz\ga$ of infinitely many copies of any sample algebra $\ga$. Here, we have shown that $\bi_\bz$ is directly acting on the monotone algebra $\gam$ since it is more convenient for applications.

The next lines are devoted to characterise the spreadable stochastic processes associated to the monotone commutation relations, which correspond to states on the monotone $C^*$-algebra
$\gam$. As we have just seen that the monoid $\bi_\bz$ acts on $\gam$ as $*$-endomorphisms, the states we are looking for are exactly those invariant under such an action. In addition, they are included in the collection of shift invariant states, as follows from \eqref{scaprz}: {\it i.e.} each spreadable stochastic process is automatically stationary.

After denoting $\ga_o:=\text{span} \{X\in \gam_o \mid l(X)>0\}$ and $\ga$ its norm closure, from Corollary 5.10 of \cite{CFL}, $\gam$ results to be the $C^*$-algebra obtained by adding the identity to $\ga$. Then necessarily any $Y\in\gam$ is decomposed as $Y:=X+cI$, where $X\in\ga$ and $c\in\mathbb{C}$.

The state at infinity $\om_\infty$ (see {\it e.g.} \cite{BR1})
is then defined as
$$
\om_\infty(Y)=\om_\infty(X+cI):=c\,.
$$
As usual, the vacuum state $\om$ is given by the vacuum expectation values
$$
\om(Y):=\langle Y \Om,\Om\rangle\,.
$$
The following result yields the desired characterisation.
\begin{prop}
\label{inv}
The $*$-weakly compact set of spreading invariant states on $\gam$ is
$$
\cs_{\mathbb{I}_\mathbb{Z}}(\gam)=\{(1-x)\om_\infty + x\om \mid x\in [0,1]\}\,.
$$
\end{prop}
\begin{proof}
Since $\cs_{\bi_\bz}(\gam)\subset\cs_\bz(\gam)$, by Theorem 5.12 of \cite{CFL}
it is enough to check that both $\om$ and $\om_\infty$ are spreading invariant.

Fix $g\in\bi_\bz$ and denote by $\G_g$ the corresponding action on $\gam$ as a $*$-endomorphism. Concerning the state at infinity, since $\G_g(\ga)\subset\ga$, we have for $X\in\ga$, $c\in\bc$,
$$
\om_\infty\big(\G_g(X+cI)\big)=\om_\infty\big(\G_g(X)\big)+c=c=\om_\infty\big(X+cI)\,.
$$
Since any state on $\gam$ is determined by its values on the dense subalgebra $\gam_o$, for the vacuum state we consider a generic element

$$
Y:= cI+ \sum_{\l\in H} a_\l X_\l+\sum_{k\in K} b_ka_ka^\dagger_k\,,
$$
where $H$ and $K$ are finite parts of $\G$ and $\mathbb{Z}$, respectively, with the condition that if $\l\in H$ and $\l=(\{i\},\{j\})$, then $i\neq j$. One has $\om(X_\l)=0$ for any $\l\in H$. Moreover, from \eqref{beta1} one gets
$\om\big(\G_g(X_\l)\big)=0$
as well, and $\om\big(\G_g(a_ia^\dagger_i)\big)=\om(a_ia^\dagger_i)$.
Thus we finally obtain
$$
\om\big(\G_g(Y)\big)=c+\sum_{i\in I} b_i=\om(Y)\,.
$$
\end{proof}

\section{the exchangeability for monotone processes}
\label{sec5}

As a direct application of the Hamel basis we have found for $\gam_o$, the goal of the present section is that to
determine the stochastic processes arising from monotone commutation relations ({\it i.e.} states on the monotone algebra $\gam$ in our terminology) which are invariant under the natural action of the permutation group on their finite joint distributions. Among the class of such monotone stochastic processes, the last are precisely the exchangeable ones.

When doing this, one promptly comes across an obstruction. It concerns the fact that the group of the permutations does not act in a natural way ({\it i.e.} by an action coming from the permutation of the underlying indices)
through positive maps on $\gam$. This forces us to work either on the universal object made of the free product $C^*$-algebra, or possibly directly on the $*$-algebra $\gam_o$, perhaps using well defined suitable linear maps useful to test exchangeability. The next lines are devoted to show that the latter choice is available and easier than the former one.

Let now $\s\in \bp_{\mathbb{Z}}$, and $I:=\{i_1,\ldots, i_m\}$ a finite ordered subset in $\mathbb{Z}$, that is $i_1<\cdots < i_m$. We say $\s$ is order preserving on $I$, $OP(I)$ being the shorthand notation, if $\s(i_h)<\s(i_{h+1})$ for each $h=1,\ldots, m-1$.

For the elements $X_{(\l_1,\l_2)}$ and $\s\in\bp_\bz$, we define
\begin{equation*}
T_\s\big(X_{(\l_1,\l_2)}\big):=\left\{
\begin{array}{ll}
             X_{(\s(\l_1),\s(\l_2))} & \text{if}\,\, \s\in OP(\l_1)\cap OP(\l_2)\,, \\
             0 & \text{otherwise}\,.
           \end{array}
           \right.
\end{equation*}
It is easy to see that the $T_\s$, $\s\in\bp_\bz$ extend by linearity to $*$-maps acting on $\gam_o$. Unfortunately, these linear maps are not positive. Indeed, let $i\in\mathbb{Z}$, $X_1:=a_ia^\dag_i$, $X_2:=a^\dag_{i+1}$, $\pi_i$ being the transposition exchanging $i$ and $i+1$, and finally $\xi:=e_{i+2}-2e_i\otimes e_{i+2}$. By means of Lemma 5.4 in \cite{CFL}, one has
\begin{align*}
&\big\langle T_{\pi_i}((X_1+X_2)^*(X_1+X_2))\xi,\xi\big\rangle \\
=&\big\langle T_{\pi_i}(a_{i+1}+a_{i+1}a^\dag_{i+1}+a_ia^\dag_i+a^\dag_{i+1})\xi, \xi\big\rangle \\
=&\big\langle (a_i+a_ia^\dag_i+a_{i+1}a^\dag_{i+1}+a^\dag_i) \xi,\xi\big\rangle
=-2\,.
\end{align*}
In addition, they do not implement any action of $\bp_\mathbb{Z}$. Indeed, consider a $\l$-form $X_{(\l_1,\l_2)}$ with length at least 2, and $\s\in \bp_\mathbb{Z}$ such that $\s\notin OP(\l_1)\cap OP(\l_2)$. Then
$$
X_{(\l_1,\l_2)}=T_{\s^{-1}\s}\big(X_{(\l_1,\l_2)}\big)\neq T_{\s^{-1}}\circ T_\s\big(X_{(\l_1,\l_2)}\big)=0\,.
$$
As previously stressed, this does not prevent us to classify exchangeable stochastic processes. In fact, fix a state $\f$ on $\gam$, which is nothing but a stochastic process we are dealing with, and consider its GNS representation $(\pi_\f,\ch_\f,\Om_\f)$. We first recall that the algebra of the samples $\ga$ appearing in Section \ref{sec2.2} is isomorphic to the unital $*$-algebra generated by $a_i$ for each $i\in\bz$, and therefore
$\ga\sim\bm_2(\bc)$ with matrix-units $e_{11}=aa^\dagger$, $e_{12}=a$, $e_{21}=a^\dagger$, and finally $e_{22}=a^\dagger a$.
In addition, the maps $\{\iota_j\}_{j\in\bz}$ are determined by
$$
\iota_j(a)=\pi_\f(a_j)\,,\quad j\in\bz\,.
$$
By using the Hamel basis, it is enough to check the exchangeability on the linear generators. Thus, fix $\s\in\bp_\bz$ together with
$X_{(\{i_1,\dots,i_m\},\{j_n,\dots,j_1\})}$
with length $l\big(X_{(\{i_1,\dots,i_m\},\{j_n,\dots,j_1\})}\big)>2$, the remaining cases being trivial. If $(\l_1,\l_2):=(\{i_1,\dots,i_m\},\{j_n,\dots,j_1\})$, the exchangeability condition \eqref{exzc} is easily written as
\begin{align*}
\f\big(X_{(\l_1,\l_2)}\big)=&\big\langle\pi_\f\big(X_{(\l_1,\l_2)}\big)\Om_\f,\Om_\f\big\rangle \\
=&\big\langle\iota_{i_1}(a^\dagger)\cdots\iota_{i_m}(a^\dagger)\iota_{j_1}(a)\cdots\iota_{j_n}(a)\Om_\f,\Om_\f\big\rangle\\
=&\big\langle\iota_{\s(i_1)}(a^\dagger)\cdots\iota_{\s(i_m)}(a^\dagger)\iota_{\s(j_1)}(a)\cdots\iota_{\s(j_n)}(a)\Om_\f,\Om_\f\big\rangle\\
=&\big\langle\pi_\f\big(T_\s\big(X_{(\l_1,\l_2)}\big)\big)\Om_\f,\Om_\f\big\rangle=\f\big(T_\s\big(X_{(\l_1,\l_2)}\big)\big)\,.
\end{align*}
We are then ready to state the following
\begin{prop}
\label{symm}
If $\f\in\cs(\gam)$ satisfies
$$
\f\lceil_{\gam_o}\circ T_\s=\f\lceil_{\gam_o}\,,\,\,\text{for all}\,\,\s\in\bp_\bz\,,
$$
then
$$
\f=(1-x)\om_\infty + x\om,\,\,\,\,\, \text{for some}\,\,\,\,\, x\in [0,1]\,.
$$
\end{prop}
\begin{proof}
We first notice the above convex combination is invariant under any $T_\s$, since the state at infinity and the vacuum state are both invariant.

Indeed, for the latter we first take $X_{(\l_1,\l_2)}$ an arbitrary element of the basis, and distinguish two cases.

\noindent 1) $(\l_1,\l_2)=(\emptyset,\emptyset)$, or $(\l_1,\l_2)=(\{i\},\{i\})$ for some integer $i$. Since any permutation $\s$ belongs to $OP(\l_1)\cap OP(\l_2)$, here one has
$$
\om\big(X_{(\l_1,\l_2)}\big)=1=\om\big(T_{\s}\big(X_{(\l_1,\l_2)}\big)\big)\,.
$$

\noindent 2) $(\l_1,\l_2)\neq(\emptyset,\emptyset)$, and $(\l_1,\l_2)\neq(\{i\},\{i\})$. Now $X_{(\l_1,\l_2)}$ is a $\l$-form with length at least 1, and then
$$
\om\big(X_{(\l_1,\l_2)}\big)=0\,.
$$
In this case, if $\s\in OP(\l_1)\cap OP(\l_2)$, one has that $T_\s\big(X_{(\l_1,\l_2)}\big)$ is a $\l$-form with the same length of $X_{(\l_1,\l_2)}$, and therefore $\om\big(T_{\s}\big(X_{(\l_1,\l_2)}\big)\big)=0$. When instead $\s\notin OP(\l_1)\cap OP(\l_2)$, one directly has $T_\s\big(X_{(\l_1,\l_2)}\big)=0$.
As a consequence, by linearity $\om\lceil_{\gam_o}\circ T_\s=\om\lceil_{\gam_o}$.

Now we show that if $\f\lceil_{\gam_o}$ is invariant under all $T_\s$, $\s\in\bp_\bz$, then $\f$ is shift invariant. In this case, the assertion will follow from Theorem 5.12 of \cite{CFL}. Indeed, denote $\f$ one of such states. By continuity of the shift, we can reduce again the matter to the Hamel basis of $\gam_o$. Take a generic element $X_{(\l_1,\l_2)}$. Then there exists a cycle $\s\in\bp_\bz$ (depending on the chosen element) such that, first $\s\in OP(\l_1)\cap OP(\l_2)$ trivially, and then
$X_{\t(\l_1),\t(\l_2)}=X_{\s(\l_1),\s(\l_2)}$. Denoting by $\a$ action of the one step shift, we get
\begin{align*}
\f\big(\a(X_{(\l_1,\l_2)})\big)=&\f(X_{\t(\l_1),\t(\l_1)})=\f(X_{\s(\l_1),\s(\l_1)})\\
=&\f\big(T_{\s}(X_{(\l_1,\l_2)})\big)=\f(X_{(\l_1,\l_2)})\,.
\end{align*}
\end{proof}
We end the section by noticing that it would be desirable to extend the maps $T_\s$ on the whole $C^*$-algebra as bounded maps. This step relies in proving
\begin{equation*}
\|T_\s\|<+\infty\,,\quad \s\in\bp_\bz\,.
\end{equation*}
On the other hand, its proof seems to be a quite complicated task, which perhaps is not needed for our analysis on exchangeability.

\bigskip

\section*{Acknowledgements}
\noindent The authors acknowledge the support of the italian INDAM-GNAMPA. The first and the second named authors kindly acknowledge also the Department of Physics of the University of Pretoria, where part of the results of the present paper were obtained, and in particular R. Duvenhage
for the hospitality and financial support. Finally, the authors are grateful to an anonymous referee whose comments improved the presentation of the paper.

\end{document}